\documentclass[12pt]{extarticle}
\usepackage[T1]{fontenc}
\usepackage[latin2]{inputenc}
\usepackage[english]{babel}
\usepackage[left=2.54cm,right=2.54cm,top=2.54cm,bottom=2.54cm]{geometry}
\usepackage{scrextend}
\changefontsizes[16pt]{14pt}
\usepackage{dsfont, complexity, lmodern,wrapfig,hyperref} 
\usepackage{amsmath, amsthm}
\usepackage{amssymb}
\usepackage{float}
\usepackage{graphicx}
\usepackage{xcolor}
\usepackage{paralist}
\usepackage{refcount}

\newtheorem{theorem}{Theorem}[]

\newtheorem*{theorem*}{Theorem}
\newtheorem*{obs*}{Observation}

\begin{document}
\title{On graphs and their application to \\determinant and set theory\thanks{Translated from~\cite{Kon16} by \'{A}gnes Cseh, Universit\"{a}t Bayreuth. Remarks and corrections are very welcome at \texttt{agnes.cseh@uni-bayreuth.de}.}}
\author{
D\'{e}nes K\H{o}nig\\\bigskip Budapest, October 1915}
\date{}
\maketitle
\begin{abstract}
	The presented work focuses on problems from determinant theory, set theory and topology. The term \emph{graph} is the binding element that connects these problems. Graphs are distinguished by their geometrical simplicity, which helps in showing the equivalence between various seemingly unrelated problems, besides providing solutions to several open questions discussed here.
\end{abstract}
\thispagestyle{empty}
\setcounter{page}{1}
\pagestyle{plain}
\section{Graphs~\cite{Pet91}}
\label{sec:1}

We are given a finite number of points and some pairs of these points are connected by one or more (but finitely many) edges. A figure that arises from such a construction is in general called \emph{graph}. A graph is \emph{connected} if one can reach any of its vertices\footnote{Only the initially fixed points of the graph are called vertices.} from any other vertex via edges. If the graph is disconnected, then it can be decomposed into connected subgraphs in a unique manner. The advantage of a certain algebraic application of graphs lies in the triviality of this statement. If every vertex is connected to the same number of edges, then the graph is called \emph{regular} and the constant number of edges at each vertex is called the \emph{degree} of the regular graph. The so-called \emph{bipartite graphs} play a highly distinguished role. A graph belongs to this class  if every closed line drawn along its edges (walk) consists of an even number of edges. An example for this is the edge system of the cube.

\begin{obs*}
	A graph is bipartite if and only if its vertices can be partitioned into two classes so that every edge runs between vertices in different classes.
\end{obs*}

\begin{proof}Indeed if such a partitioning is possible, then every closed walk visits vertices in the first and second vertex classes in an alternating manner, thus the closed walk must consist of an even number of edges. The other direction is also true: one can assign each vertex along a walk in an alternating manner to classes one and two. We run into trouble only if we walked along a closed walk comprising an odd number of edges. This also holds for disconnected graphs, but the classes are not uniquely defined in this case.\end{proof}

One can define disconnected and bipartite graphs in a similar manner, because disconnected graphs can be characterized as follows. Their vertices can be partitioned into two classes so that no two points in the \emph{same} class are connected via an edge.

We call the set of edges $G_k$ of a graph $G$  a \emph{factor of $k$th degree} of $G$ if every vertex of $G$ is incident to exactly $k$ edges in~$G_k$. A factor of $k$th degree is itself a regular graph of degree $k$ and its vertices coincide with the vertices of the original graph~$G$. A set of edges form a factor of \emph{first} degree if every vertex is incident to exactly one edge in the set. If $G$ is regular of degree $n$, while $G_k$ is regular of degree $k$, then the edges outside of $G_k$ form a regular graph $G_{n-k}$ of degree $n-k$. In this case it follows from Petersen~\cite{Pet91} that $G = G_k G_{n-k}$.

The decomposition of a regular graph into more than two factors can be defined analogously. The sum of the degree of factors always equals the degree of the original graph. In the following, we concentrate on decomposing a regular graph into factors of first degree.

Every graph of second degree consists of simple (non-crossing) closed walks. This has a factor of first degree if and only if all these closed walks consist of an even number of edges~\cite[page 195]{Pet91}. As an extension of this fact we will show the following theorem.

\begin{theorem} Every regular bipartite graph has a factor of first degree.\\
	\text{\normalfont This statement is a direct consequence of the following one, claiming even more.}\footnote{If we assume Theorem~A) to be true, then the graph $G_k$ of degree $k$ has a factor $G_1$ of first degree, moreover, $G = G_1 G_{k-1}$. Similarly, $G_{k-1} = G'_1 G_{k-2},G_{k-2} = G''_1 G_{k-3}$, and so on, where $G', G''$ and so on are also factors of first degree. Finally, $G = G_1 G'_1 \dots G^{k-1}_1$, which is indeed a decomposition into factors of first degree. Theorem~B) for even $k$ follows directly from a statement of Petersen~\cite[page 200]{Pet91}. It would not be challenging to reduce the case of odd $k$ to the same statement. The proof presented here is simpler and does not need case distinction, moreover it leads to more general results.}
	\end{theorem}
	\begin{theorem} Every regular bipartite graph of degree $k$ can be decomposed into $k$ factors of first degree.\\
	\text{\normalfont Again, this statement is a consequence of the following.}
	\end{theorem}
	\begin{theorem} If each vertex of a bipartite graph is incident to at most $k$ edges, then it is possible to assign each edge of the graph an index between 1 and $k$ so that any two adjacent edges have different indexes.
\end{theorem}

Theorem~B) is indeed a direct consequence of Theorem~C), because in a regular graph of degree $k$ the edges of the same index form a factor of first degree.

\begin{proof}
We use induction to prove Theorem~C). Even though this method fails for Theorem~B), it leads to success in the case of Theorem~C). If the number of edges in $G$ is at most $k$, then the statement holds trivially. Thus we assume that it holds for graphs with less than $N$ edges and now we show correctness for graphs with $N$ edges.

If we omit an arbitrary edge $e$ from $G$, then we derive the bipartite graph~$G'$. Due to the assumption in the induction, $G'$ can be decomposed  in the wished way. The used indexes are now fixed. Say that the omitted edge $e$ runs between vertices $A$ and $B$. The first case is that there is an index that occurs in the neighborhood of neither $A$ nor~$B$. In this case we can assign this skipped index to $e$ and we are done. Thus we can now concentrate on the second case, namely that there is an index (say, \texttt{1}) missing from $B$, which occurs at~$A$. At $B$ there is certainly a missing index, because in $G'$ both $A$ and $B$ have at most $k-1$ edges incident to them. Let index \texttt{2} be a missing index at~$A$ and let edge $AA_1$ be the edge incident to $A$, which was assigned the index \texttt{1}. There is possibly an edge $A_1 A_2$ with index \texttt{2}, an edge $A_2 A_3$ with index \texttt{1}, an edge $A_3 A_4$ with index \texttt{2} and so on. We now create this alternating path $A, A_1, A_2, A_3, \dots$ as long as possible. Due to the indexes \texttt{1,2} this path is uniquely defined. We cannot reach the same vertex twice on this path, because then $A_i$ would be the first vertex reached for a second time and then the indexes \texttt{1,2} occur at least three times at~$A_i$. We also cannot end in $A$, because index \texttt{2} is missing from there and index \texttt{1} has already been used at edge~$AA_1$. Finally, we cannot reach $B$, because this could only happen via an edge with index \texttt{2} -- index \texttt{1} is missing at~$B$. Then, the path from $A$ to $B$ would consist of an even number of edges. This path together with the omitted edge $e$ would then form a cycle of odd length in $G$, which we assumed not to exist.

We derived that $A, A_1, A_2, A_3, \dots, A_r$ is a non-crossing alternating \texttt{1-2} path with two open ends, moreover $B$ is not on the path. We now swap the indexes \texttt{1} and \texttt{2} and leave the indexes in the rest of the graph unchanged. The modified indexes satisfy our requirements. This can be seen immediately for the starting vertex $A$, which had no edge with index \texttt{2}. The claim is trivial for the inner vertices $A_i$ and as of the end vertex $A_r$, it had exactly one of the indexes \texttt{1,2}, because otherwise the path would not have ended at~$A_r$. With this operation we have now reached out goal to assign $e$ an index, because index \texttt{1} does not occur at $A$ either.

With this we have proven Theorem~C) and thus Theorems~A) and B) too.
\end{proof}

Before we turn to the applications of these results in determinant theory and set theory we would like to point out that our Theorem~B) is closely related to a well-known problem in topology, namely the problem of \emph{map coloring}. The still open \emph{4-color-conjecture} states the following. The countries on a map can be colored in such a way that no two countries share the same color if they have a common line of border. Obviously, it is sufficient to consider the case when the borders in the map form a regular graph of degree~3. In this case the 4-color-conjecture is equivalent to the theorem of Tait~\cite[page 413]{Wer04}, which states that such a graph can be decomposed into factors of first degree. From this it does not follow that every graph of degree~3 can be decomposed into factors of first degree, because there are two necessary conditions for a graph to represent a planar map. 
\begin{enumerate}
	\item  It can be drawn into the plane (or sphere) so that no edge crossing occurs. 
	\item Each of the graph edges belongs to a crossing-free closed walk.
\end{enumerate}
 The theorem of Tait does not hold without these two restrictions, as one can show it on two counterexamples\footnote{The first counterexample is a remarkable graph of Petersen~\cite{Pet98}, the other counterexample is for instance a graph of degree~3 from Sylvester, which was also published as Figure 11 in the more often cited work of Petersen.}. Tait formulated his theorem without assuming these two properties and claimed it to be trivial~\cite{Tai84}. On the other hand, if the graph is bipartite, then the theorem holds even without the two conditions and this can be seen from our Theorem~B). Yet this delivers no new result about map coloring, because Kempe~\cite{Kem79} has already shown that three colors already suffice. In order to prove the 4-color-conjecture in the general case, it would certainly be beneficial to give a simple combinatorial characterization of graphs satisfying the first condition. The polyhedron theorem of Euler would certainly play an important role in such a result\footnote{This I have attempted in two papers published in Hungarian~\cite{Kon11a,Kon11b}}.

\section{Application to determinants}

Bipartite graphs and matrices on integer or real elements can be connected through a property of determinants, which depends only on the absolute value of the members\footnote{Translator's notice: A \emph{member} of a determinant is defined as $\prod_{i=1}^n{a_{i \sigma(i)}}$, where $\sigma$ is a permutation of the set $\{ 1,2,\dots,n\}$. Notice that according to the Leibnitz formula $\det(A) = \sum_{\sigma \in S_n}{sgn(\sigma)}\prod_{i=1}^n {a_{i \sigma(i)}}$.} and not on their sign.

First we consider a matrix $D= |a_{ik}|, i, k = 1, 2, \dots, n$ where the elements $a_{ik}$ are non-negative integers. The square matrix $D$ defines the following bipartite graph~$G$. The rows of $D$ represent vertices $A_1, A_2, \dots, A_n$, while the columns of $D$ represent vertices $B_1, B_2, \dots, B_n$ in~$G$. There are $a_{ik}$ edges running between vertices $A_i$ and~$B_k$. Notice that $a_{ik}$ might be 0 as well. Vertices within $A$ or within $B$ are never connected. The constructed graph is indeed bipartite, because every walk visits vertices in $A$ and $B$ in an alternating manner. Vertex $A_i$ is incident to as many edges as the sum of the elements in the $i$th row and similarly, $B_k$ is incident to as many edges as the sum of the elements in the $k$th column. Thus $G$ is regular if and only if the elements in every row and column of $D$ sum up to the same value. Since no two vertices in $A$ or in $B$ are adjacent, we can construct the following factor of first degree: $K_i = \{A_1 B_{i_1}, A_2 B_{i_2}, \dots, A_n B_{i_n}\}$, where $(i_1, i_2, \dots, i_n)$ is a permutation of the integers $1,2, \dots, n$. This $K_i$ is a factor of first degree of $G$ if and only if none of $a_{1i_1}, a_{2i_2}, \dots,a_{ni_n}$ is 0, and thus the product $a_{1i_1} \cdot a_{2i_2}\cdot \dots \cdot a_{ni_n}$ also cannot be~0. This product is exactly a general member of $D$, not considering the sign of it. Since Theorem~A) states that every regular graph has a factor of first degree, we have now arrived to the following theorem.

\begin{theorem} If each row and column sum of a determinant consisting of non-negative [integers] real numbers equals the same positive number, then at least one member of the determinant is non-zero.
\end{theorem}

This theorem has been proved for integer matrices, but it is easy to see that it holds for matrices consisting of rational or real numbers. On the other hand, if fails if negative integers occur, such as in the following matrix, in which each row and column sum is 1, yet its determinant has 6 members of 0.

\[
  \begin{bmatrix}
    0 & 0 & 1 \\
		0 & 0 & 1 \\
		1 & 1 & -1 \\
  \end{bmatrix}
\]

Now we will consider the special case where all nonzero members of the determinant are exactly~1. In this case we can reformulate the condition on row and column sums, namely it is equivalent to each row and column containing the same number of 0 elements. Each graph factor of first degree corresponds to a member of the determinant then. Since the members of the determinant are zero or non-zero independently from the value of the non-zero elements, we derive the following theorem.

\begin{theorem}
If the number of non-zero elements in each each row and column is exactly $k$, then the determinant has at least $k$ non-zero members.

(These $k$ non-zero members can be chosen so that every non-zero element of the matrix occurs in exactly one non-zero member.)
\end{theorem}

For $k=1$ it is trivial to decide whether there are exactly $k$ non-zero members or more. The same question for the case $k=2$ can be answered with the help of the graph corresponding to the matrix. Interestingly, the case $k=2$ leads to the question of how many factors of first degree a regular graph of degree~3 contains. This seems to be much more challenging and it is also closely connected to the 4-color-conjecture.

Theorems~B) and E) can also be interpreted in the following manner.

\begin{theorem}
 Assume that we have $kn$ figures standing on a quadratic table of $n^2$ fields, so that one field might contain more than one figure. This configuration can be seen as a combination of $k$ such figure distributions, where each row and column contains one figure at most.
\end{theorem}

An analogous interpretation can also be given to Theorem~C).

\section{Infinite graphs and application to set theory}

The presented analysis of graphs does not require any geometric interpretation of graphs. So far, we have always talked about $G$ as the set of its edges, this is, a set of pairs $(AB, CD, AE, \dots)$. These pairs were formed out of elements $(A,B,C \dots)$, called vertices. In order to model parallel edges between the same two vertices, one needs to permit elements with multiplicity in the edge set. All of our geometric terms (connected graph, regular graph, degree and closed walk in a graph) can be defined in an abstract manner, ignoring geometric representation and purely relying on the set~$G$. This observation becomes crucial once we allow sets $(A,B,C \dots)$ and $(AB, CD, AE, \dots)$ to be infinite, even to be of arbitrary cardinality. We remark that nothing stands in the way to define graph-theoretical terms for infinite graphs, moreover, it is not necessary to define regularity, degree, factors, walks and so on separately for infinite graphs. We remind the reader though that a graph is connected if between any two of its vertices there is a finite walk\footnote{It might be useful to note that the statement 'a vertex can be reached via an infinite walk from another vertex' does not make sense at all. Therefore, our phrasing 'between any two of its vertices there is a finite walk' can be replaced by 'between any two of its vertices there is a walk'. Thus, every closed walk of a graph consists of a finite number of edges.}. Without any geometrical interpretation one can prove the statement that even an infinite graph can be uniquely decomposed into connected components.

The notion of graphs is extended greatly if we do not restrict the size of the vertex and edge sets of a graph. There is only one restriction we still keep, namely we only consider graphs where the number of edges meeting in a vertex remains under a finite bound. We call a graph \emph{countable} if its edge set has a possibly infinite, but countable cardinality. Here we investigate countable graphs besides the finite graphs already studied and start with the following theorem. 

\begin{theorem}
 If in $G$ the number of edges meeting in any vertex remains under a finite bound $h$, then $G$ can be decomposed into finite and countable components\footnote{It is sufficient to assume that the number of edges meeting in any vertex is finite.}.
\end{theorem}

\begin{proof} It is sufficient to show that if $G$ is connected on the top of the assumption of the theorem, then $G$ is finite or countable. This is clear. From an arbitrary vertex $P$ we can reach at most $h$ vertices through a walk consisting of one edge, at most $h^2$ vertices through a walk consisting of two edges and at most $h^{\nu}$ vertices through a walk consisting of $\nu$ edges. Thus the set of vertices that are reachable from $P$ via a finite walk, this is, the set of all vertices in $G$ is finite or countable. From this follows that the number of edges is also finite or countable.
\end{proof}

Right after factors of first degree (a set of non-connected edges), factors of second degree are the simplest amongst infinite graphs. These can be decomposed into finite closed walks and walks of countably many edges, which are infinite at both ends. If a factor of second degree is bipartite, then clearly each component of it is a union of two factors of first degree. Therefore\footnote{This consequence is invalid for those who do not accept Zermelo's axiom of choice~\cite{Zer04}.} so is the entire graph. Thus, Theorems~A) and B) carry over to factors of second degree.

We now give examples for infinite graphs of higher degree. The quadratic grid of the plane represents an infinite factor of degree~4, while its 3-dimensional variant in the space is an infinite factor of degree~6.

Our Theorem~G) becomes important once we apply graphs in set theory. In many cases we can reduce problems to finite and countable sets, as we will soon demonstrate. We now use graphs\footnote{Using graphs, one can give a simple and graphic proof for the equivalence theorem of set theory, which was also proved by Bernstein. The proof of J. K\H{o}nig~\cite{Kon06} is the simplest known and it is indeed based on graph theory. There, only graphs of degree~2 are used, thus it seems to be superfluous to use the terminology of graphs, due to the minimalistic setting. I derived the results presented in this paper as I had attempted to use my father's methods for the above mentioned theorem of Bernstein.} to prove an extension of the following theorem of Bernstein\footnote{Bernstein's dissertation was defended in G\"{o}ttingen in 1901 and it was published as~\cite{Ber05}. He proves the case $\nu=2$ only and already this proof is very complicated. The proof presented in this paper deals with this case using straightforward intuitions. Bernstein's theorem is a direct consequence of Zermelo's well-ordering theorem. Our proof might be of interest even for those who accept Zermelo's results, even though we did not attempt to avoid Zermelo's axiom of choice. We only use the most basic notions of set theory, such as set, element, projection, while the concept of an ordering does not play a role at all.}.

\begin{theorem*}
	Let $m$ and $n$ be arbitrary cardinalities and let $\nu$ be a finite number. From $\nu m = \nu n$ follows that $m = n$.
\end{theorem*}

This theorem claims the following. Assume that sets $M$ and $N$ have cardinality $m$ and $n$, respectively, and we substitute every element of those by $\nu$ different elements. If the  equality $\nu m = \nu n$ holds for the cardinality of the two sets derived from $M$ and $N$, then sets $M$ and $N$ were also equivalent. The theorem can also be formulated without using the notion of cardinality.

\begin{theorem}
	If two sets are in a reversible $(1,\nu)$-relation to each other, then they are equivalent.
\end{theorem}

First we define the term 'reversible $(1,\nu)$-relation' formally. We say that $M$ is in a $(1,\nu)$-relation to $N$ if every element of $M$ corresponds to $\nu$ elements of~$M$. These $\nu$ elements of~$M$ do not need to be different, because to each $M$-element we assign each of the corresponding $N$-elements with a multiplicity. Thus the $\nu$ elements of $N$ corresponding to an element $a \in M$ can be listed as $(b_1, b_2, \dots, b_{\mu})$, where $\mu \leq \nu$. These elements are assigned multiplicity $(s_1, s_2, \dots, s_{\mu})$, respectively, so that $s_1 + s_2 + \dots+ s_{\mu} = \nu$. Such a relation from $M$ to $N$ is called \emph{reversible} if the corresponding relation from $N$ to $M$ is also a $(1,\nu)$-relation. In other words, a reversible $(1,\nu)$-relation possesses the property that if $b \in N$ is assigned to $a \in M$, then $a$ is also assigned to $b$ with the same multiplicity.

\begin{proof} This setting allows us to represent the reversible $(1,\nu)$-relation of two sets $M$ and $N$ through bipartite regular graphs. Since $M$ and $N$ have cardinality $m$ and $n$, respectively, this is the same as proving the correctness of the equality $\nu m = \nu n$. We define graph $G$ as follows. The elements of $M$ and $N$ are $G$'s vertices, and a vertex in $M$ is connected to a vertex in $N$ via $s$ edges, where $s$ is the multiplicity of the assignment between the two elements. Two vertices within $M$ or within $N$ are never connected by an edge. The constructed graph is indeed bipartite, because every closed walk visits vertices in $M$ and $N$ in an alternating manner. The graph is also regular, because each of its vertices is incident to $\nu$ edges.

According to Theorem~G), this graph $G$ can be decomposed into finite and countable components. Let $G_1$ be an arbitrary component of $G$, containing vertex sets $M_1$ and $N_1$ with cardinality $m_1$ and $n_1$, respectively. We know that $\nu m_1 = \nu n_1$, because both side of the equation correspond to the cardinality of the edge set of~$G_1$. Since $m_1$ and $n_1$ are finite or countable, and for this case the theorem of Bernstein is trivial, we derived $m_1 = n_1$. The same argument can be repeated for every component of $G$. Adding these up we arrive to $m = n$, which proves the theorem of Bernstein.
\end{proof}

If we want to generalize the theorem of Bernstein (Theorem~H)), we stumble upon serious difficulties. We would like to emphasize that we have not succeeded in proving such a generalization.

\begin{theorem}
	If two sets are in a reversible $(1,\nu)$-relation to each other, then there is a bijection that assigns elements to each other if and only if they are assigned to each other in the given reversible $(1,\nu)$-relation.
\end{theorem}

Two elements are assigned to each other if and only if the corresponding vertices are connected via an edge in the above defined graph. Thus, Theorem~I) simply claims that $G$ has a factor of first degree. To prove it, it would be sufficient to show the correctness of 
Theorem~A) for infinite graphs. Even more, according to Theorem~G) it is only necessary to prove Theorem~A) for countable graphs. The other direction is also true: assuming Theorem~I), we derive Theorem~A) for infinite graphs. This is because every bipartite graph of degree $\nu$ can be seen as a graph representing the reversible $(1,\nu)$-relation of two sets. As we have already shown it for finite graphs, the vertices of infinite graphs can also be partitioned into two sets so that every edge runs between vertices in different sets\footnote{If the graph can be decomposed into infinitely many components, then once again we need to assume Zermelo's axiom of choice.}. 

With this we have shown the equivalence of Theorem~I) and Theorem~A) for finite and infinite graphs. Theorems~B) and C) are equivalent to these as well. This is because Theorem~B) is a consequence of Theorem~C) and it can be seen easily for infinite graphs that Theorem~A) also follows from Theorem~B) and finally, the reverse direction holds as well, namely that Theorem~B) is a consequence of Theorem~A) and Theorem~C) is a consequence of Theorem~B). Only the very last observation requires a proof.

\begin{theorem}
Theorem~C) is a consequence of Theorem~B).
\end{theorem}

\begin{proof}
 Let $G$ be an arbitrary bipartite graph with the property that every vertex of it is incident to at most $k$ edges. We create a new graph $H$ by adding vertices and edges to~$G$. To each vertex in $G$ we create a new vertex and connect two new vertices with the same amount of edges that connect their original copies in~$G$. Moreover, each vertex in $G$ is connected to its new copy via $k -\alpha$ edges, where $\alpha \leq k$ is the number of edges that are incident to the vertex in~$G$. We do not add any more vertices or edges. The created graph $H$ is regular, because each of its vertices is incident to $\alpha + (k- \alpha) = k$ edges. Moreover, $H$ is a bipartite graph, which can be proved as follows. Since $G$ was a bipartite graph, its vertices form sets $I$ and $II$, so that every edge in $G$ connects vertices in different sets. Now we order the new points to sets $I$ and $II$ depending on which set their original copy belongs to, ensuring that the edges connecting the original and new copies run between vertices in different sets. We now assume Theorem~B) and observe that $H$ can be decomposed into $k$ factors of first degree. Depending on which factor it belongs to we assign an index between $1$ and $k$ to each edge. This assignment of indexes corresponds to the claim in Theorem~C).
\end{proof}

Finally, we show that it would be sufficient to prove the equivalent Theorems~A), B), C) and I) in the case where the degree in Theorem~A) and B), $k$ in Theorem~C) and $\nu$ in Theorem~I) is a prime number. For Theorem~A, this follows from Theorem~K), and thus it holds for the remaining three theorems as well. 

\begin{theorem}
	If every bipartite graph of degree $\mu$ and every bipartite graph of degree $\nu$ has a factor of first degree, then every bipartite graph of degree $\mu \nu$ also has a factor of first degree.
\end{theorem}
\begin{proof}
Assume that every bipartite graph of degree $\mu$ and every bipartite graph of degree $\nu$ has a factor of first degree and let $P$ be an arbitrary vertex of an arbitrary bipartite graph $G$ of degree $\mu \nu$. We substitute $P$ by $\mu$ vertices $P_1, P_2, \dots, P_{\mu}$ and reroute the $\mu \nu$ edges of $P$ to run to one of $P_1, P_2, \dots, P_{\mu}$ so that each $P_i$, $(i = 1, 2, \dots, \mu)$ is incident to exactly $\nu$ of these edges. Having executed this operation to every vertex of $G$ we derive a graph $G_{\nu}$ of degree~$\nu$. This is a bipartite graph, because every closed walk in $G_{\nu}$ corresponds to a closed walk in $G$ with the same set of edges. According to our assumptions, $G_{\nu}$ has a factor $G_1$ of first degree. We now unite the vertices $P_1, P_2, \dots, P_{\mu}$ to $P$ and get the original graph $G$ back. This transforms $G_1$ into a factor $G_{\mu}$ of degree $\mu$. As $G_{\mu}$ is a subgraph of $G$, it is also bipartite and it has a factor of first degree. This is also a factor of first degree in $G$ and thus we have shown Theorem~K). 
\end{proof}

The following, partly more general result can be shown by repeating the first part of the above argumentation almost word-by-word.

\begin{theorem*}
	If every bipartite graph of degree $\nu$ has a factor of degree $k$, then every bipartite graph of degree $\mu \nu$ has a factor of degree $\mu k$.
\end{theorem*} 

We have shown Theorems~A) and B) for degree~2 even for infinite graphs. It follows from Theorem~K) that Theorems~A), B), C) and I) are certainly correct if the degree, $k$ in Theorem~C) and $\nu$ in Theorem~I) is a power of~2. Our methods do not carry over for arbitrary degree, as a matter of fact, not even for degree~3, even though we can restrict ourselves to countable graphs as we have seen it. The induction we used in the proof of Theorem~C) fails immediately if the graph has infinitely many edges.

\medskip

The results in this work were submitted to the Hungarian Academy of Sciences on November 15, 1915.
\vspace{-0.5cm}
\bibliographystyle{abbrv}
\bibliography{mybib}

\begin{thebibliography}{10}

\bibitem{Ber05}
F.~Bernstein.
\newblock Untersuchungen aus der {M}engenlehre.
\newblock {\em Mathematische Annalen}, 61(1):117--155, 1905.

\bibitem{Kem79}
A.~B. Kempe.
\newblock On the geographical problem of the four colours.
\newblock {\em American Journal of Mathematics}, 2(3):193--200, 1879.

\bibitem{Kon11a}
D.~K\H{o}nig.
\newblock Vonalrendszerek k\'{e}toldal\'{u} fel\"{u}leteken.
\newblock {\em Mathematikai es Term\'{e}szettudom\'{a}nyi \'{E}rtes\'{i}t\H{o}}, 29-1:112--117, 1911.

\bibitem{Kon11b}
D.~K\H{o}nig.
\newblock A vonalrendszerek nemsz\'{a}m\'{a}r\'{o}l.
\newblock {\em Mathematikai es Term\'{e}szettudom\'{a}nyi \'{E}rtes\'{i}t\H{o}}, 29-3:345--350, 1911.

\bibitem{Kon16}
D.~K{\H{o}}nig.
\newblock {\"U}ber {G}raphen und ihre {A}nwendung auf {D}eterminantentheorie und {M}engenlehre.
\newblock {\em Mathematische Annalen}, 77(4):453--465, 1916.

\bibitem{Kon06}
J.~K{\H{o}}nig.
\newblock Sur le th{\'e}orie des ensemble.
\newblock {\em Comptes Rendus Hebdomedaire des S{\'e}ances de l'Academie des Science, Paris}, 143:110--2, 1906.

\bibitem{Pet91}
J.~Petersen.
\newblock Die {T}heorie der regul{\"a}ren {G}raphs.
\newblock {\em Acta Mathematica}, 15(1):193--220, 1891.

\bibitem{Pet98}
J.~Petersen.
\newblock Sur le th{\'e}oreme de {T}ait.
\newblock {\em L'interm{\'e}diaire des Math{\'e}maticiens}, 5:225--227, 1898.

\bibitem{Tai84}
P.~Tait.
\newblock Listing's {T}opologie ({I}ntroductory address to the {E}dinburgh {M}athematical {S}ociety, {N}ovember 9, 1883).
\newblock {\em Philosophical Magazine, January}, 1884.

\bibitem{Wer04}
P.~Wernicke.
\newblock {\"U}ber den kartographischen {V}ierfarbensatz.
\newblock {\em Mathematische Annalen}, 58(3):413--426, 1904.

\bibitem{Zer04}
E.~Zermelo.
\newblock Beweis, da{\ss} jede {M}enge wohlgeordnet werden kann.
\newblock {\em Mathematische Annalen}, 59(4):514--516, 1904.

\end{thebibliography}
\newpage

\begin{wrapfigure}{r}{0.26\textwidth}
    \includegraphics[width=0.26\textwidth]{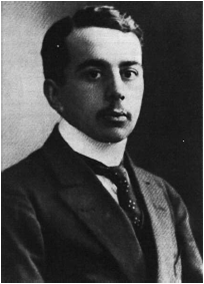}
\end{wrapfigure}

		\textbf{D\'enes K\H{o}nig} (September 21, 1884 - October 19, 1944) was a Hungarian mathematician of Jewish origin, who worked in and wrote the first textbook on the field of graph theory.
		\smallskip

K\H{o}nig was born in Budapest, as the son of mathematician Gyula K\H{o}nig. In 1899, while still attending high school he published his first work in the journal 'Matematikai \'es Fizikai Lapok'. After his graduation in 1902, he won first place in a mathematics competition 'E\"otv\"os Lor\'and'. He spent four semesters in Budapest and his last five in G\"ottingen, during which he studied under J\'ozsef K\"ursch\'ak and Hermann Minkowski.
\smallskip

He received his doctorate in 1907 due to his dissertation in geometry. The same year he began working for the 'Technische Hochschule in Budapest' (today Budapest University of Technology and Economics) and remained a part of the faculty till his death in 1944. At first he started as an assistant in problem sessions, in 1910 he was promoted to 'Oberassistant', and then promoted to 'Privatdocent' in 1911 teaching nomography, analysis situs (later to be known as topology), set theory, real numbers and functions, and graph theory (the name 'graph theory' didn't appear in the university catalog until 1927). During this time he would be a guest speaker giving mathematics lecture for architecture and chemistry students. In the years 1918, 1920 and 1936 K\H{o}nig  published three textbooks. He became a full professor in 1935.
\smallskip

From 1915 to 1942 he was on a committee to judge school contests in mathematics, collecting problems for these contests, and organizing them. 
K\H{o}nig's activities, lectures and his book 'Theorie der endlichen und unendlichen Graphen' played a vital role in the growth of graph theoretical work of L\'aszl\'o Egyed, P\'al Erd\H{o}s, Tibor Gallai, Gy\"orgy Haj\'os, J\'ozsef Kraus, Tibor Szele, P\'al Tur\'an, Endre V\'azsonyi, and many others. He marked the beginning of graph theory as its own branch of mathematics.
\smallskip

K\H{o}nig, although of Jewish origin, was brought up a Christian and so was not in danger from the  persecution of Hungarian Jews for a while. In fact he worked to help Jewish mathematicians. On October 15, 1944 the German troops invaded Hungary. Days later on October 19, 1944 he committed suicide to evade persecution from the Nazis. 
\begin{flushright}
{\color{gray}Source: Wikipedia and http://www-history.mcs.st-and.ac.uk/}
\end{flushright}

\end{document}